\documentclass[a4paper, 12pt]{article}

\usepackage[sort&compress]{natbib}
\bibpunct{(}{)}{;}{a}{}{,} 

\usepackage{amsthm, amsmath, amssymb, mathrsfs, multirow, url, subfigure}
\usepackage{graphicx} 
\usepackage{ifthen} 
\usepackage{amsfonts}
\usepackage[usenames]{color}
\usepackage{fullpage}

\theoremstyle{plain} 
\newtheorem{thm}{Theorem}
\newtheorem{cor}{Corollary}

\theoremstyle{definition}

\theoremstyle{remark} 
\newtheorem*{astep}{A-step}
\newtheorem*{pstep}{P-step}
\newtheorem*{cstep}{C-step}

\newcommand{\prob}{\mathsf{P}}

\newcommand{\bel}{\mathsf{bel}}
\newcommand{\pl}{\mathsf{pl}}
\newcommand{\pval}{\mathsf{pval}}

\newcommand{\bin}{{\sf Bin}}
\newcommand{\unif}{{\sf Unif}}
\newcommand{\nm}{{\sf N}}

\newcommand{\chisq}{{\sf ChiSq}}

\newcommand{\RR}{\mathbb{R}}
 
\newcommand{\U}{\mathscr{U}}

\newcommand{\XX}{\mathbb{X}}
\newcommand{\UU}{\mathbb{U}}

\newcommand{\TT}{\mathbb{T}}

\newcommand{\Xbar}{\bar X}

\newcommand{\Ubar}{\overline{U}}

\DeclareMathOperator{\cl}{cl}

\renewcommand{\S}{\mathcal{S}}
\renewcommand{\SS}{\mathbb{S}}

\renewcommand{\phi}{\varphi} 
\newcommand{\eps}{\varepsilon}

\title{A note on p-values interpreted as plausibilities}
\author{
Ryan Martin \\
Department of Mathematics, Statistics, and Computer Science \\
University of Illinois at Chicago \\
\url{rgmartin@uic.edu} \\
\mbox{} \\
Chuanhai Liu \\
Department of Statistics \\
Purdue University \\
\url{chuanhai@purdue.edu}
}
\date{\today}

\begin{document}

\maketitle 

\begin{abstract}  
P-values are a mainstay in statistics but are often misinterpreted.  We propose a new interpretation of p-value as a meaningful plausibility, where this is to be interpreted formally within the inferential model framework.  We show that, for most practical hypothesis testing problems, there exists an inferential model such that the corresponding plausibility function, evaluated at the null hypothesis, is exactly the p-value.  The advantages of this representation are that the notion of plausibility is consistent with the way practitioners use and interpret p-values, and the plausibility calculation avoids the troublesome conditioning on the truthfulness of the null.  This connection with plausibilities also reveals a shortcoming of standard p-values in problems with non-trivial parameter constraints.   

\smallskip

\emph{Keywords and phrases:} Hypothesis test; inferential model; nesting; plausibility function; predictive random set.
\end{abstract}

\section{Introduction}
\label{S:intro}

P-values are ubiquitous in applied statistics, but are widely misinterpreted as either a sort of Bayesian posterior probability that the null hypothesis is true, or as a frequentist error probability.  Indeed, in 2012, media were reporting the discovery of the elusive Higgs boson particle \citep{overbye2012} and statistics blogs were pointing out how some journalists and physicists had misinterpreted the resulting p-values.  Our objective here is to provide a new and simpler way to understand them, so that these misinterpretations might be avoided. 

A prime reason for the frequent misinterpretation of p-values is that the standard textbook definition is inconsistent with people's common sense.  The goal of this paper is to provide a more user-friendly interpretation.  We show that the p-value can be interpreted as a plausibility that the null hypothesis is true.  This ``plausibility'' is precisely defined within the inferential model (IM) framework of \citet{imbasics}, built upon two fundamental principles of \citet{sts.discuss.2014} for valid and efficient statistical inference.  Consider the problem of testing a null hypothesis $H_0$ versus a global alternative $H_1$.  We show that, under mild conditions, for any p-value (depending on $H_0$ and the choice of test statistic), there exists a valid IM such that the plausibility of $H_0$ is the p-value.  In this sense, the p-value can be understood as the plausibility, given the observed data, that $H_0$ is true.  In the Higgs boson report, since the p-value is minuscule ($p \ll 10^{-6}$) one concludes that the hypothesis $H_0: \text{``the Higgs boson does not exist''}$ is highly implausible, hence, a discovery.  This line of reasoning based on small p-values is consistent with Cournot's principle \citep[e.g.,][]{shafer.vovk.2006}.   

The word ``plausibility'' fits the way practitioners use and interpret p-values: a small p-value means $H_0$ is implausible, given the observed data.  Evaluating plausibility involves a probability calculation that does not require one to assume that $H_0$ is true, so one avoids the questionable logic of proving $H_0$ false by using a calculation that assumes it is true.  The use of IMs to provide probabilistic interpretations of classically non-probabilistic summaries is proving to be beneficial; see, for example, \citet{randset}.  

The remainder of the paper is organized as follows.  Section~\ref{S:fisher} sets up our notation and gives the formal definition of p-value, with a brief discussion of its common correct and incorrect interpretations.  The basics of IMs are introduced in Section~\ref{S:infmodel}, in particular, predictive random sets and plausibility functions.  In Section~\ref{S:pval} we prove that, given essentially any hypothesis testing problem, there is a valid IM such that the corresponding plausibility function, evaluated at the null hypothesis, is the p-value.  There we highlight a similar connection between the IM plausibilities and objective Bayes posterior probabilities, and an apparently unrecognized shortcoming of p-values in problems with non-trivial parameter constraints.  Two examples involving binomial and normal data are presented in Sections~\ref{SS:binomial}--\ref{SS:normal}, and some concluding remarks are given in Section~\ref{S:discuss}.

\section{The p-value}
\label{S:fisher}

\subsection{Setup and formal definition}
\label{SS:setup}

Let $X$ denote observable data, taking values in $\XX$.  There is a sampling model $\prob_{X|\theta}$, indexed by a parameter $\theta \in \Theta$, and the goal is to make inference on $\theta$ using the observed data $X=x$.  Here both $X$ and $\theta$ are allowed to be vector-valued, but do will not make this explicit in the notation.  The hypothesis testing problem starts with a hypothesis, or assertion, about the unknown $\theta$.  Mathematically, this is characterized by a subset $\Theta_0 \subset \Theta$, and we write $H_0: \theta \in \Theta_0$ for the null hypothesis and $H_1: \theta \not\in \Theta_0$ for the alternative hypothesis.  The goal is to use observed data $X=x$ to determine, with some measure of certainty, whether $H_0$ or $H_1$ is true.  

Consider the description of the p-value given by \citet[][p.~39]{fisher1959}, viewed as follows.  If the observed $X=x$ gives small p-value, then one of two things occurred: relative to $H_0$, a rare chance event has occurred, \emph{or} $H_0$ is false.  The unlikeliness of the former drives us to conclude the latter.  To put this in more standard terms, suppose there is a test statistic $T: \XX \to \RR$, possibly depending on $\Theta_0$, such that large values of $T(X)$ suggest that $H_0$ may not be true.  The p-value is defined, for $X=x$, as
\begin{equation}
\label{eq:pval}
\pval(x) = \pval_{T,\Theta_0}(x) = \sup_{\theta \in \Theta_0} \prob_{X|\theta}\{T(X) \geq T(x)\}. 
\end{equation}
When $\Theta_0 = \{\theta_0\}$, a point null, \eqref{eq:pval} simplifies to $\pval(x) = \prob_{X|\theta_0}\{T(X) \geq T(x)\}$, the expression found in most introductory textbooks.   

Intuitively, $\pval(x)$ compares the observed $T(x)$ to the sampling distribution of $T(X)$ when $H_0$ is true.  If $\pval(x)$ is small, then $T(x)$ is an outlier under $H_0$ and we conclude that $H_0$ is implausible.  Conversely, if $\pval(x)$ is relatively large, then the observed $T(x)$ is consistent with at least one $\prob_{X|\theta}$, with $\theta \in \Theta_0$, so $H_0$ is plausible in the sense that it provides an acceptable explanation of reality.

\subsection{Standard interpretations}

Standard textbooks have adopted an equivalent though arguably more obscure interpretation.  The standard textbook interpretation of p-value goes something like this: 

\begin{quote}
$\pval(x)$ is the probability that an observable $X$ is ``at least as extreme'' as the $x$ actually observed, assuming $H_0$ is true.
\end{quote}

This leads to the common misinterpretation of p-value as a sort of Bayesian posterior probability of $H_0$.  \citet[][Sec.~3.3]{lehmann.romano.2005}, after laying out the details of the Neyman--Pearson testing program, have it that

\begin{quote}
$\pval(x)$ is the greatest lower bound on the set of all $\alpha$ such that the size-$\alpha$ test rejects $H_0$ based on $T(x)$.
\end{quote}
A danger here is that the conditioning on $H_0$ is hidden in the definition of size; users can potentially misinterpret $\pval(x)$ as the probability of incorrectly rejecting $H_0$ based on $x$.  

Some statisticians have abandoned the use of p-values, advocating for other tools for measuring evidence supporting $H_0$ and/or testing $H_0$, such as confidence intervals; see, e.g., \citet[][Sec.~4.3]{berger.delampady.1987} and the discussion by G.~Casella and R.~Berger on that same paper.  This preference for confidence intervals is fairly common in medical, social, and other applied sciences.  Some journals, such as the \emph{American Journal of Public Health}, have made concerted efforts to get authors to use confidence intervals rather than p-values \citep{fidler.etal.2004}.  Still, confidence intervals are not free of their own difficulties, nor are Bayes factors \citep{kass.raftery.1995} or Bayesian p-values \citep{gelman.meng.stern.1996, rubin1984}.  We think a better or simpler way to understand the ubiquitous p-value is a valuable contribution.

\section{Review of inferential models}
\label{S:infmodel}

\subsection{Big picture}
\label{S:picture}

The inferential model (IM) framework produces exact prior-free probabilistic measures of evidence for/against any assertion about the unknown parameter; see \citet{imbasics}, \citet{mzl2010}, and \citet{zl2010}.  This is accomplished by first making an explicit association between the observable data $X$, the unknown parameter $\theta$, and an unobservable auxiliary variable $U$.    Random sets are introduced to predict the unobservable $U$, and inference about $\theta$ is obtained via probability calculations with respect to the distribution of this random set.  The IM framework has some connections with existing approaches, such as fiducial \citep{hannig2009, hannig.lee.2009, hannig2012}, confidence distributions \citep{xie.singh.strawderman.2011, xie.singh.2012}, Dempster--Shafer theory \citep{dempster2008, shafer1976, shafer2011}, generalized p-values and confidence intervals \citep{tsui.weerahandi.1989, weerahandi1993, chiang2001}, and Bayesian inference with default, reference, and/or data-dependent priors \citep{berger2006, bergerbernardosun2009, mghosh2011, fraser2011, fraser.reid.marras.yi.2010}.  

IMs, fiducial, and Dempster--Shafer theory all introduce auxiliary variables into the inference problem.  Both fiducial and Dempster--Shafer theory condition on the observed $X=x$, then develop a sort of distribution on the parameter space by inverting the data--parameter--auxiliary variable relationship and assuming that $U$ retains its \emph{a priori} distribution after $X=x$ is observed.  The IM approach targets the (unattainable) best possible inference corresponding to the case where $U$ is observed.  Uncertainty about $\theta$, after $X=x$ is observed, is propagated from the uncertainty about hitting the true $U$ with a random set.  In addition to accomplishing Fisher's goal of prior-free probabilistic inference, IMs produce inferential output which is valid for any assertion of interest (Section~\ref{SS:valid}); fiducial probabilities are valid only for special kinds of assertions \citep[][Sec.~4.3.1]{imbasics}.  Moreover, a general theory of optimal IMs, concerning efficiency of the resulting inference, may not be out of reach.

\subsection{Construction}
\label{SS:construct}

Following \citet{imbasics}, the IM construction proceeds in three steps.   

\begin{astep}
This proceeds by specifying an {association} between $X$, $\theta$, and $U$.  Like fiducial, Dempster--Shafer, and Fraser's structural models \citep{fraser1968}, this can be described by a pair $(\prob_U, a)$, where $\prob_U$ describes the distribution (and also, implicitly, the support $\UU$) of the auxiliary variable $U$, and $a$ describes the data-generating mechanism driven by $\prob_U$.  We write this as 
\[ X = a(\theta, U), \quad \text{with} \quad U \sim \prob_U. \]
That is, if $U$ is sampled from $\prob_U$ and then plugged in to the function $a$ for a given $\theta$, then the resulting $X$ has distribution $\prob_{X|\theta}$.  The association need not be described by a formal equation---it is enough to have a rule/recipe to construct $X$ from a given $\theta$ and $U$; see e.g., Section~\ref{SS:binomial}.  To complete the A-step, construct a sequence of subsets of $\Theta$ indexed by $(x,u)$:
\begin{equation}
\label{eq:focal1}
\Theta_x(u) = \{\theta: x=a(\theta,u)\}. 
\end{equation} 
\end{astep}

\begin{pstep}
Based on the idea that knowing the unobserved value of $U$ is ``as good as'' knowing $\theta$ itself, the goal of the {prediction} step is to predict this unobserved value with a predictive random set, denoted by $\S$.  Certain assumptions are required on the support $\SS$ and distribution $\prob_\S$ of the predictive random set; see Section~\ref{SS:valid}.  
\end{pstep}

\begin{cstep}
This step combines the observed $X=x$, which specifies a sub-collection of sets $\Theta_x(\cdot)$ in \eqref{eq:focal1}, with the predictive random set $\S$.  The result is an $x$-dependent random subset of $\Theta$:
\begin{equation}
\label{eq:focal2}
\Theta_x(\S) = \textstyle\bigcup_{u \in \S} \Theta_x(u).  
\end{equation}
Evidence for/against an assertion $A \subseteq \Theta$ concerning the unknown parameter can now be obtained via the $\prob_\S$-probability that $\Theta_x(\S)$ is/is not a subset of $A$.  More precisely, we evaluate 
\begin{equation}
\label{eq:belief}
\bel_x(\cdot;\S) = \prob_\S\{\Theta_x(\S) \subseteq \cdot\},
\end{equation}
the belief function, at both $A$ and $A^c$, as a summary of evidence for and against $A$, respectively.  Alternatively, we can report $\bel_x(A;\S)$ together with  
\begin{equation}
\label{eq:plaus.fun}
\pl_x(A;\S) = \prob_\S\{\Theta_x(\S) \cap A \neq \varnothing\} = 1-\bel_x(A^c;\S), 
\end{equation}
the plausibility function at $A$.  It can be readily shown that $\bel_x(A;\S) \leq \pl_x(A;\S)$ for any $A$ and any $\S$.  Then, as described briefly below, the pair $\{ \bel_x(\cdot;\S), \pl_x(\cdot;\S) \}$ is used for inference about $\theta$; see \citet{imbasics} for details.  
\end{cstep}

Statistical inference based on the IM output focuses on the relative magnitudes of $\bel_x(A;\S)$ and $\pl_x(A;\S)$.  An assertion $A$ is deemed true (resp.~untrue), given $X=x$, if both $\bel_x(A;\S)$ and $\pl_x(A;\S)$ are large (resp.~small).  Conversely, if $\bel_x(A;\S)$ is small and $\pl_x(A;\S)$ is large, then there is no clear decision to be made about the truthfulness of $A$, given $X=x$, so maybe one needs more data.  The definition of ``small'' and ``large'' values of belief/plausibility functions are specified by the theoretical validity properties discussed in Section~\ref{SS:valid}.  

One can also construct frequentist procedures based on plausibility functions.  For example, for $\alpha \in (0,1)$, a $100(1-\alpha)$\% plausibility region for $\theta$ is defined as 
\begin{equation}
\label{eq:plaus.region}
\Pi_\alpha(x) = \{\theta: \pl_x(\theta; \S) > \alpha\}. 
\end{equation}
It is a consequence of Theorem~\ref{thm:valid} below that this region achieves the nominal $1-\alpha$ frequentist coverage probability.    

Throughout it is assumed that $\Theta_x(u)$ in \eqref{eq:focal1} satisfies $\Theta_x(u) \neq \varnothing$ for all $(x,u)$ pairs.  This boils down to there being no non-trivial constraints on the parameter space $\Theta$.  When this fails, one can usually take a dimension-reduction step, described in \citet{imcond}, to transform to a problem where this assumption holds.  Under this condition, it is sometimes convenient to evaluate the plausibility on the $u$-space as opposed to the $\theta$-space as in \eqref{eq:plaus.fun}.  Given $x$ and $A$, let 
\begin{equation}
\label{eq:a.event}
\UU_x(A) = \cl\{u: \Theta_x(u) \subseteq A\},
\end{equation}
where $\cl B$ denotes the closure of the set $B$.  If $\Theta_x(u) \neq \varnothing$ for all $(x,u)$, as we have assumed, then belief and plausibility can be evaluated on the $u$-space as 
\begin{equation}
\label{eq:plaus.fun.alt}
\bel_x(A;\S) = \prob_\S\{\S \subseteq \UU_x(A)\} \quad \text{and} \quad \pl_x(A;\S) = 1-\prob_\S\{\S \subseteq \UU_x(A^c)\}. 
\end{equation}
This formulation is used in the main result in Section~\ref{S:pval}.

\subsection{IM validity}
\label{SS:valid}

It is important that the IM's belief and plausibility functions are meaningful across similar studies.  This type of meaningfulness is referred to as validity in \citet{imbasics}.  Here, the IM is said to be valid if
\begin{equation}
\label{eq:valid2}
\sup_{\theta \in A} \prob_{X|\theta} \{\pl_X(A;\S) \leq \alpha\} \leq \alpha, \quad \forall \; A \subseteq \Theta, \quad \forall \; \alpha \in (0,1). 
\end{equation}
This means that, if $A$ is true, then $\pl_x(A;\S)$ is small for only a small proportion of possible $x$ values, ``outliers.''  Since it holds for all $A \subseteq \Theta$, a similar statement about $\bel_X(A;\S)$ can also be made.  Validity holds, without special modification, even for the scientifically important case of singleton $A$.  In fact, for reasonably chosen predictive random sets \citep[see][Corollary~1]{imbasics}, the latter ``$\leq \alpha$'' can be replaced by ``$=\alpha$;'' hence $\pl_X(A;\S) \sim \unif(0,1)$ when $A=\{\theta_0\}$ is true.  In Theorem~\ref{thm:pval} below we show that the p-value is a plausibility function at the null hypothesis.  So \eqref{eq:valid2} restates the familiar result that, if the null hypothesis is true, then the p-value is (stochastically dominated by) a $\unif(0,1)$ random variable. 

The validity property \eqref{eq:valid2} holds if the predictive random set $\S$ satisfies certain conditions, no requirements on $\prob_{X|\theta}$ or the association $(\prob_U,a)$ are needed.  Let $(\UU,\U)$ be the measurable space on which $\prob_U$ is defined, and assume that $\U$ contains all closed subsets of $\UU$.  \citet{imbasics} gives the following result.

\begin{thm}
\label{thm:valid}
The IM is valid for all assertions $A \subseteq \Theta$ if $\Theta_x(u) \neq \varnothing$ for all $(x,u)$ and the predictive random set $\S$ satisfies the following:
\vspace{-2mm}
\begin{itemize}
\item[\rm P1.] The support $\SS \subset 2^\UU$ of $\S$ contains $\varnothing$ and $\UU$, and: \\
{\rm (a)} each $S \in \SS$ is closed and, hence $\prob_U$-measurable, and \\
{\rm (b)} for any $S,S' \in \SS$, either $S \subseteq S'$ or $S' \subseteq S$.
\vspace{-2mm}
\item[\rm P2.] The distribution $\prob_\S$ of $\S$ satisfies $\prob_\S\{\S \subseteq K\} = \sup_{S \in \SS: S \subseteq K} \prob_U(S)$, $K \subseteq \UU$.  
\end{itemize}
\end{thm}

\citet{imbasics} show that a wide variety of predictive random sets are available for which P1--P2 hold, so that IM validity is rather easy to arrange.  However, efficiency is a concern and, for this, they present a theory of optimal IMs.

\section{P-value as an IM plausibility}
\label{S:pval}

\subsection{Main result}
\label{SS:main}

On the association $(a, \prob_U)$, the null hypothesis $\Theta_0$, and the test statistic $T: \XX \to \TT$, we assume the following.
\begin{itemize}
\item[A1.] For every $(x,u)$ there exists $\theta$ such that $T(x) = T(a(\theta,u))$.
\vspace{-2mm}
\item[A2.] $\sup_{\theta \in \Theta_0} T(a(\theta,\cdot))$ is a $\prob_U$-measurable function.
\vspace{-2mm}
\item[A3.] $\prob_U\{ \sup_{\theta \in \Theta_0} T(a(\theta,U)) < t\} = \inf_{\theta \in \Theta_0} \prob_U\{T(a(\theta,U)) < t\}$ for all $t \in \TT$.
\end{itemize}
Here, A2 ensures the meaningfulness of the probability statement in A3, and holds generally under mild separability and continuity conditions, respectively, on $\Theta_0$ and on $T$ and $a$, while A3 makes precise the stochastic smoothness and stochastic ordering $T(X)$ should possess as a function of $\theta$.  Assumptions~A2--A3 hold trivially for the important point-null case.  It is also easy to check A3 in many common examples, e.g., if $X_1,\ldots,X_n$ are iid $\nm(\theta,1)$, and $T(X) = \Xbar$, then $T(a(\theta,U)) = \theta + \Ubar$, and A3 holds for any $\Theta_0$ of the form $(-\infty, \theta_0]$.  

\begin{thm}
\label{thm:pval}
If \emph{A1--A3} hold for the given association $(a,\prob_U)$, hypothesis $\Theta_0$, and test statistic $T: \XX \to \TT$, then there exists an admissible predictive random set $\S$, depending on $T$ and $\Theta_0$, such that the plausibility function $\pl_x(\Theta_0;\S)$ equals $\pval(x)=\pval_{T,\Theta_0}(x)$ in \eqref{eq:pval} for all $x \in \XX$.  
\end{thm}

\begin{proof}
Without loss of generality, we reduce the baseline association $X=a(\theta,U)$, with $U \sim \prob_U$, to $T(X) = T(a(\theta,U))$, again with $U \sim \prob_U$.  In this case, the A-step of the IM construction generates the collection of subsets:
\[ \Theta_x(u) = \{\theta: T(x) = T(a(\theta,u))\}, \quad x \in \XX, \quad u \in \UU. \]
These sets are non-empty for all $(x,u)$ by A1.  For the P-step, we define a collection $\SS=\{S_t: t \in \TT\}$ of subsets of $\UU$ with
\[ S_t = \cl\{u: \textstyle\sup_{\theta \in \Theta_0} T(a(\theta,u)) < t\}, \quad t \in \TT.  \]
The sets are closed, are nested, and $\prob_U$-measurability follows from A2.  Thus P1 in Theorem~\ref{thm:valid} holds.  Define a predictive random set $\S$, supported on $\SS$, with distribution $\prob_\S$ satisfying
\begin{equation}
\label{eq:t.star}
\prob_\S\{\S \subseteq K\} = \prob_U(S_{t_K^\star}) = \inf_{\theta \in \Theta_0} \prob_{X|\theta}\{T(X) < t_K^\star\}, \quad K \subseteq \UU, 
\end{equation}
where $t_K^\star = \sup\{t \in \TT: S_t \subseteq K\}$; the last equality in \eqref{eq:t.star} is a consequence of A3.  For such as $\S$, the corresponding IM is valid.  For notational consistency, set $A=\Theta_0$.  The C-step proceeds as in the general case in Section~\ref{SS:construct}, and, for any $x \in \XX$, the plausibility function \eqref{eq:plaus.fun.alt}, evaluated at $A$, satisfies
\begin{equation}
\label{eq:containment}
\pl_x(A; \S) = 1-\prob_\S\{\S \subseteq \UU_x(A^c)\} = 1-\prob_\S\{\S \subseteq S_{T(x)}\}. 
\end{equation}
The second equality in \eqref{eq:containment} needs justification.  First, we have $S_{T(x)} \subseteq \UU_x(A^c)$ since 
\begin{align*}
u \in S_{T(x)} & \implies T(x) > \textstyle\sup_{\theta \in A} T(a(\theta,u)) \\
& \implies T(x) = T(a(\theta,u)) \quad \exists \; \theta \not\in A \\
& \implies \Theta_x(u) \subseteq A^c \\
& \implies u \in \UU_x(A^c).
\end{align*}
It remains to show that $S_{T(x)}$ is the largest of the $S_t$'s contained in $\UU_x(A^c)$.  For any small $\eps > 0$, there exists $u \in S_{T(x)+\eps}$ such that $T(x) \leq \sup_{\theta \in A} T(a(\theta,u))$; for this $u$, $\Theta_x(u) \not\subseteq A^c$, so $u \not\in \UU_x(A^c)$.  We have verified \eqref{eq:containment}, so 
\begin{align*}
\pl_x(A; \S) & = 1-\prob_\S\{\S \subseteq S_{T(x)}\} \\
& = 1-\inf_{\theta \in A} \prob_{X|\theta}\{T(X) < T(x)\} \quad \text{[by \eqref{eq:t.star}]} \\
& = \sup_{\theta \in A} \prob_{X|\theta}\{T(X) \geq T(x)\}. 
\end{align*}
The right-hand side is $\pval(x)$ in \eqref{eq:pval}, completing the proof. 
\end{proof}

\begin{cor}
\label{co:pval}
Under \emph{A1}, if $\Theta_0=\{\theta_0\}$, then the conclusion of Theorem~\ref{thm:pval} holds.  
\end{cor}

\begin{proof}
Conditions A2--A3 hold automatically for singleton $\Theta_0$ and suitable $T$. 
\end{proof}

\subsection{Remarks}
\label{SS:remarks}


\citet[][p.~375]{dempster2008} points out a similar connection between plausibility and p-value; specifically, he shows numerically how Fisher's p-value can be decomposed into two pieces---one piece corresponding to belief in $H_0$ and the other corresponding to ``don't know''---the sum of which is our plausibility.  His example is for the standard test for a Poisson mean based on a one-sided alternative hypothesis, and he claims no such a correspondence in general.  

In the Bayesian setting, a search for ``objective'' priors often focuses on probability matching \citep[e.g.,][]{mghosh2011}, i.e., choose the prior such that the corresponding posterior tail probabilities and p-values are asymptotically equivalent.  Given the connection between p-values and IM plausibilities, these objective  Bayes posterior probabilities can also be interpreted as plausibilities.  This is perhaps not surprising given that objective Bayes posterior distributions can be viewed as a simple and attractive way to approximate frequentist p-values \citep{fraser2011}.  

This connection between p-values and plausibilities also casts light on the argument in \citet{schervish1996} concerning the use of p-values as measures of evidence; see, also, \citet{bergersellke1987}.  He shows that, in general, p-values fail to satisfy that, for a given $x$, if $\Theta_0' \subseteq \Theta_0$, then the p-value for $\Theta_0'$ should be no more than the p-value for $\Theta_0$.  Theorem~\ref{thm:pval} explains this lack of coherence in that p-values for different hypotheses may be plausibilities with respect to different IMs with different scales.  The same is true for Bayesian posterior probabilities for $\Theta_0'$ and $\Theta_0$ if different priors are used for each testing problem, which would not necessarily be unusual.    

In case $\Theta_x(u) = \varnothing$ for some pair $(x,u)$, constructing an IM with plausibility function matching the p-value cannot be done as described in the proof of Theorem~\ref{thm:pval}.  Such a situation arises, for example, in a normal mean problem $\nm(\theta,1)$ with $\Theta=[-1,1]$.  If $X=-1$ is observed, then $\Theta_{-1}(u) = \{\theta \in [-1,1]: -1 = \theta + \Phi^{-1}(u)\}$ is empty for $u > 1/2$.  For such problems, \citet{leafliu2012} present a modification of the IM approach that stretches the predictive random set just enough so that $\Theta_x(\S)$ is not empty while maintaining validity.  The result of this stretching is, in general, an increase in the plausibility function.  The p-value depends only on the null hypothesis, so is not affected by parameter constraints.  This is a shortcoming of the p-value, as evidence for a particular assertion should automatically become larger when the range of possible alternatives shrinks.

\subsection{Binomial example}
\label{SS:binomial}

Consider a binomial model, $X \sim \bin(n,\theta)$, where $n$ is a known positive integer and $\theta \in (0,1)$ is the unknown success probability.  Inference on $\theta$ in the binomial model is a basic problem that is far from trivial \citep[e.g.,][]{bcd2001}.  In this case, the natural association is 
\[ F_\theta(X-1) \leq U < F_\theta(X), \quad U \sim \unif(0,1), \]
where $F_\theta$ is the $\bin(n,\theta)$ distribution function.  There is no simple equation linking $(X, \theta, U)$ in this case, just a rule ``$X=a(\theta,U)$'' for producing $X$ with given $\theta$ and $U$.  We construct the p-value-based IM for a one-sided assertion/hypothesis.  

Consider $A=(0,\theta_0]$ for some fixed $\theta_0 \in (0,1)$.  If the null hypothesis is $H_0: \theta \in A$, then the uniformly most powerful test rejects $H_0$ in favor of $H_1: \theta \in A^c$ if and only if $T(X)=X$ is too large.  With this choice of $T$, for the A-step, we write 
\[ \Theta_x(u) = \{\theta: T(x) = T(a(\theta,u))\} = \{\theta: F_\theta(x-1) \leq u < F_\theta(x)\}. \]
If $G_{a,b}$ denotes the ${\sf Beta}(a,b)$ distribution function, then we may rewrite $\Theta_x(u)$ as 
\begin{align*}
\Theta_x(u) & = \{\theta: G_{n-x+1,x}(1-\theta) \leq u < G_{n-x,x+1}(1-\theta)\} \\
& = \{\theta: 1-G_{n-x+1,x}^{-1}(u) \leq \theta < 1-G_{n-x,x+1}^{-1}(u)\}.
\end{align*}
For the P-step, we construct the support $\SS = \{S_t: t \in \TT\}$, where, in this case, $\TT=\XX=\{0,1,\ldots,n\}$.  It is easy to see that 
\[ S_t = \cl\{u: \textstyle\sup_{\theta \in A} T(a(\theta,u)) < t\} = [F_{\theta_0}(t), 1]. \]
When equipped with the measure $\prob_\S$ in P2, determined by $\prob_U = \unif(0,1)$, the C-step produces a plausibility function for $A=(0,\theta_0]$, at the observed $X=x$, given by
\[ \pl_x(A;\S) = 1-F_{\theta_0}(x), \]
which is exactly the standard p-value for the one-sided test in a binomial problem.

\subsection{Normal variance example}
\label{SS:normal}

Consider a normal model, $\nm(\mu,\sigma^2)$, and a sequence of independent samples $X_1,\ldots,X_n$.  Here, $\theta=(\mu,\sigma^2)$ is unknown, but the goal is inference on $\sigma^2$, with $\mu$ a nuisance parameter.  Following the general conditioning principles in \citet{imcond}, we can focus on IMs determined by the minimal sufficient statistic, 
\[ \Xbar = \mu + \sigma n^{-1/2} Z \quad \text{and} \quad (n-1)S^2 = \sigma^2 W, \]
where $Z \sim \nm(0,1)$ and $W \sim \chisq(n-1)$, independent.  This association involves two auxiliary variables but, since the goal is inference about the scalar $\sigma^2$, we can reduce the dimension.  Write
\[ \Xbar = \mu + \frac{S}{n^{1/2}} \frac{Z}{\{W / (n-1)\}^{1/2}} \quad \text{and} \quad (n-1)S^2 = \sigma^2 W. \]
Since $\mu$ is a location parameter, it follows from \citet{immarg} that the first expression displayed above can be ignored, leaving the second as the marginal association for $\sigma^2$, which we now write as 
\[ T = \sigma^2 F^{-1}(U), \quad U \sim \unif(0,1), \]
where $T=(n-1)S^2$ and $F$ is the $\chisq(n-1)$ distribution function.  

Consider testing $H_0: \sigma^2 \leq \sigma_0^2$ versus $H_1: \sigma^2 > \sigma_0^2$.  For observed $T=t$, the standard test has p-value $\pval(t) = \prob\{T \geq t\} = 1 - F(t/\sigma_0^2)$.  It is straightforward to check that, with predictive random set $\S=[0,U)$, with $U \sim \unif(0,1)$, the plausibility function is 
\[ \pl_t(\{\sigma^2 \leq \sigma_0^2\}; \S) = \prob_\S\{\S \ni F(t/\sigma^2)\} = \prob\{U \geq F(t/\sigma_0^2)\} = 1-F(t/\sigma_0^2), \]
which is exactly the p-value.  Moreover, the predictive random set above is ``optimal'' in the sense of \citet[][Section~4.3.1]{imbasics}, which provides some IM-based explanation for this test being the standard one in the statistics literature.  

For an illustration, consider data presented in Problem~2-14 of \citet{montgomery2001} on the etch uniformity on silicon wafers taken during a qualification experiment.  In this case, the sample size is $n=20$ and the sample variance is $S^2=0.79$.  If $\sigma_0^2=1$ so the goal is testing if $\sigma^2 \leq 1$, the p-value is 0.72, and the null hypothesis is quite plausible.  More generally, we can plot the plausibility (or p-value) as a function of $\sigma_0^2$; see Figure~\ref{fig:pval}.  The horizontal line at $\alpha=0.1$ characterizes a 90\% plausibility lower bound for $\sigma^2$ defined by keeping all those $\sigma_0^2$ values with plausibility greater than 0.1; see \eqref{eq:plaus.region}.

\begin{figure}
\begin{center}
\scalebox{0.60}{\includegraphics{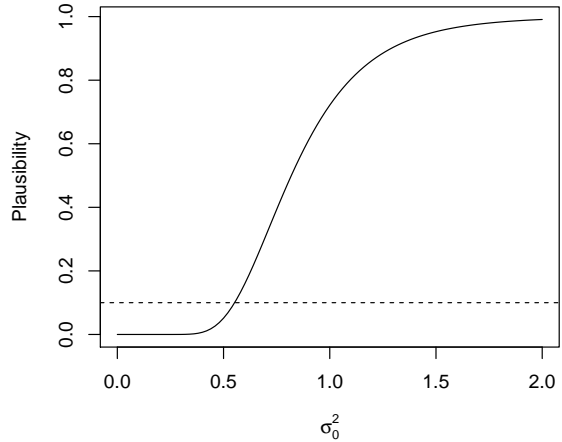}}
\end{center}
\caption{Plot of plausibility as a function of $\sigma_0^2$ in the normal variance example.}
\label{fig:pval}
\end{figure}

\section{Discussion}
\label{S:discuss}

We have developed a new user-friendly interpretation of the familiar but often misinterpreted p-value.  Specifically, we have shown that, for essentially any hypothesis testing problem, under mild conditions, there exists a valid IM such that its plausibility function, evaluated at the null hypothesis, is exactly the usual p-value.  This representation of p-values in terms of IM plausibilities casts light on a potential shortcoming of p-values that can arise in problems with non-trivial parameter constraints.  In such cases, it is not clear how to modify the p-value, while modifications of the IM plausibility are readily obtained via the methods described in \citet{leafliu2012}.


There are a numerous alternatives to p-value in the hypothesis testing literature, popular, at least in part, because of the difficulties in interpreting p-values.  For example, Jim Berger (and co-authors) have recommended converting p-values to Bayes factors, or posterior odds, for interpretation; for example, \citet{sellkebayarriberger2001} make a strong case for their suggested ``$-ep\log p$'' adjustment.  However, it is unlikely that p-values will ever disappear from textbooks and applied work, so compared to offering an alternative to the familiar p-value, it may be more valuable to offer a more user-friendly interpretation.  To borrow an analogy Larry Wasserman used on his blog:\footnote{\url{http://normaldeviate.wordpress.com/2012/07/11/} -- ``p-value police''} many people are poor drivers, but eliminating cars is not the answer to this problem. 

The connection between plausibility and p-value casts light on the nature of the IM output.  IM belief and plausibility functions are understood in \citet{imbasics} as measures of evidence given data.  The fact that, in some cases, plausibility and p-value match up is useful, suggesting that one could reason with IM plausibilities as with p-values.  The correspondence between plausibilities, p-values, and some objective Bayes posterior probabilities, and the fact that IMs contain the fiducial/Dempster--Shafer paradigms as special cases, suggests that the IM framework may in fact provide a unified perspective on robust, objective, probabilistic inference.

\section*{Acknowledgments}

The authors are grateful for helpful suggestions from the Editor, an Associate Editor, and two referees.  This work is partially supported by the U.S.~National Science Foundation, grants DMS--1007678, DMS--1208833, and DMS--1208841.

\bibliographystyle{apa}
\bibliography{/Users/rgmartin/Dropbox/Research/mybib}

\end{document}